\documentclass[12pt]{amsart}

\usepackage{vmargin}
\setmargrb{1in}{1in}{1in}{1in}
\usepackage{amstext}
\usepackage{amssymb,mathrsfs}
\usepackage{latexsym}
\newcommand\RR{{\mathbb R}}

\def\d={\,:=\,}

\newcommand{\semdir}
{\rtimes}

\font\frakten=eufm10
\newfam\frakfam
\textfont\frakfam=\frakten

\def \g{\mathfrak g}
\def \h{\mathfrak h}
\def \p{\mathfrak p}
\def \R{\mathbb R}
\def \eO {\mathcal O}
\def \b {\beta}
\def \l {\lambda}
\def \C {\mathbb C}
\def \n {\mathfrak n}
\def \zO {\Omega}
\def \S{\Sigma}
\def\O{\Omega}
\newtheorem{thm}{Theorem}
\newtheorem{lemma}[thm]{Lemma}
\newtheorem{cor}[thm]{Corollary}
\newtheorem{prop}[thm]{Proposition}
\newtheorem{Defn}[thm]{Definition}
\newtheorem{Ex}[thm]{Example}
\newtheorem{Rem}[thm]{Remark}
\newtheorem{Exs}[thm]{Examples}
\newtheorem{Rems}[thm]{Remarks}
\newtheorem{Defrem}[thm]{Definition and Remark}
\newtheorem{Remnt}[thm]{}
\newenvironment{defn}
 {\begin{Defn} \begin{rm}} {\end{rm} \hfill $\Box$ \end{Defn}}

\newenvironment{rem}
 {\begin{Rem} \begin{rm}} {\end{rm} \hfill $\Box$ \end{Rem}}

\newenvironment{prf} {{\bf Proof.}}{\hfill $\Box$}

\begin{document}

\author{B.~ Currey, H.~F{\"u}hr, K.~Taylor}


\title[Integrable wavelet transforms with abelian dilation groups]{Integrable wavelet transforms with abelian dilation groups}

\keywords{wavelet transforms; Calder\'on condition; integrable
representations; dual topology; coadjoint orbits;}
\subjclass[2000]{Primary 42C40}

\date{\today}

\begin{abstract}
We consider a class of semidirect products $G = \mathbb{R}^n \rtimes
H$, with $H$ a suitably chosen abelian matrix group. The choice of
$H$ ensures that there is a wavelet inversion formula, and we are
looking for criteria to decide under which conditions there exists a
wavelet such that the associated reproducing kernel is integrable.

It is well-known that the existence of integrable wavelet
coefficients is related to the question whether the unitary dual
of $G$ contains open compact sets. Our main general result reduces the
latter problem to that of identifying compact open sets in the
quotient space of all orbits of maximal dimension under the dual action of $H$ on $\mathbb{R}^n$.
This result is applied to study integrability for certain families
of dilation groups; in particular, we give a characterization
valid for connected abelian matrix groups acting in dimension three.
\end{abstract}

\maketitle


\section{Introduction}  \label{sect:intrt}

Before we describe the aims of this paper, we need to introduce some terminology and notations.
Let $H<{\rm GL}(n,\mathbb{R})$ denote a closed matrix group. The affine group
generated by $H$ and all translations is the semidirect product $G= \mathbb{R}^n \rtimes
H$. Elements of $G$ are denoted by pairs $(x,h) \in \mathbb{R}^n
\times H$, with group law $(x,h)(y,g) = (x+hy,hg)$. $G$ acts
unitarily on ${\rm L}^2(\mathbb{R}^n)$ via the {\em quasiregular
representation} $\pi$, defined by
\[
\pi(x,h) f(y) = |{\rm det}(h)|^{-1} f(h^{-1}(y-x))~.
\]
Given $\eta \in {\rm L}^2(G)$, the associated {\em continuous
wavelet transform} is the mapping
\[
V_\eta :{\rm L}^2(\mathbb{R}^n) \to C_b(G)~,~V_\eta f(x,h) = \langle
f,\pi(x,h) \eta \rangle~.
\] Here $C_b(G)$ denotes the space of bounded continuous functions
on $G$. Whenever $V_\eta$ is an isometry into ${\rm L}^2(G)$
(defined with respect to left Haar measure), we call $\eta$ {\em
admissible}. In turn, the group $H$ will be called {\em admissible}
if there exists an admissible vector for $\pi$.

Admissible vectors and groups have been studied extensively in the
past twenty years, with varying degrees of generality. A small and
subjective sample of the literature is
\cite{Mu,BeTa,FuMa,LWWW,Fu_LN,Fu_Calderon}.

The property we intend to investigate in this paper is {\em
integrability} of the quasi-regular representation. We call a representation
integrable if there is an admissible vector $\eta$ such that
$\Delta_G^{-1/2} V_\eta \eta$ is in ${\rm L}^1(G)$. By slight abuse of notation 
we will call such vectors $\eta$ {\em integrable admissible vectors}. 
A major motivation for studying
integrability in this context is the coorbit theory due to
Feichtinger and Gr\"ochenig \cite{FeiGr,Gr}, which heavily relies on
integrability. Coorbit spaces associated to a quasi-regular
representation can be interpreted as a family of smoothness spaces
induced by the action of the group. While the initial definition 
also relied on irreducibility of the representation, recent work
(e.g. \cite{OlChr}) suggests that a meaningful theory can also 
be developed for reducible representations, as long as integrability
(or somewhat weaker conditions) are retained. 

Besides these connections to function spaces, integrability is also
a representation-theoretic property of independent interest. For any
admissible vector $\eta$, isometry of
$V_\eta$ implies that $ V_\eta \eta = \widetilde{ V_\eta \eta} =  V_\eta \eta \ast  V_\eta \eta$.
Here we used the notation $\tilde{f} (x) = \overline{f(x^{-1})}$ for any function $f$ on $G$. If, in addition, $ F = \Delta_G^{-1/2}  V_\eta \eta \in {\rm
L}^1(G)$, this implies that $F$ is a {\em projection in the group
algebra ${\rm L}^1(G)$}, i.e. it fulfills $F = F \ast F = F^*$, where $F^*(x) = \Delta_G(x)^{-1} \overline{F(x^{-1})}$.
The question which group algebras contain projections has been studied, e.g., in
\cite{Ba,GKT,KT1,KT2,Val1,Val2}. A necessary
representation-theoretic condition for the existence of projections
in $L^1$ hails from the following observation: If $F \in {\rm
L}^1(G)$ is a projection, then ${\rm supp}(\widehat{F}) \subset
\widehat{G}$ is a compact open set. (Note that here, as in the
following, our definition of compactness does not include the
Hausdorff property.) Here $\widehat{F}$ is the group Fourier
transform of $F$, obtained by integrating $F$ against irreducible
representations, and $\widehat{G}$ denotes the unitary dual of $G$,
i.e. the equivalence classes of irreducible representations of $G$.
The topology on $\widehat{G}$ is the Fell topology, defined in terms
of uniform convergence of matrix coefficients on compact subsets of
$G$, see \cite{Di,Fo} for details. Thus, a necessary (but usually
not sufficient) condition for the existence of a projection in ${\rm
L}^1$ is the existence of a compact open set in the dual.

For the setting studied here, the most significant precursors are
\cite{GKT,KT2}; for a unified account of the results see also \cite{KT_book}. 
The paper \cite{GKT} contains a full characterization of
integrability for the case $H= \exp(\mathbb{R} A)$ (unlike our
paper, \cite{GKT} does not impose any condition on the spectrum of
$A$). Briefly, the quasiregular representation associated to such a
group is integrable iff the group is strictly contractive; or more
technically,  iff the signs of ${\rm Re} \lambda$ coincide, for all
eigenvalues $\lambda$ of $A$.

In the following, we
shall mostly concentrate on the following class of matrix groups:
$H$ is an abelian connected matrix group satisfying ${\rm spec}(h)
\subset \mathbb{R}^+$, for all $h \in H$. The chief reason for
focussing on this class is that there exist easily checked
admissibility criteria for it: $H$ is admissible iff ${\rm det}|_H
\not\equiv 1$, and in addition, the dual action of $H$ is free
almost everywhere, see \cite[Theorem 23]{BCFM}. Here, the dual action is
defined by $(h,x) \mapsto h^{-T} x$, where $h^{-T} = (h^T)^{-1}$ . We let $d = {\rm dim}(H)$.

The following theorem is the chief general result of this paper. Before we
formulate it, we need some additional notation. Orbits under $H^T$
are submanifolds of $\mathbb{R}^n$. If a matrix group $H$ is fixed,
we define, for $i \in \mathbb{N}$:
\[
\mathcal{O}_i = \{ x \in \mathbb{R}^n : {\rm dim}(H^T x) = i \}~.
\] Clearly, $\mathcal{O}_i = \emptyset$ whenever $i> d$.
For the class of groups that we consider here, the dual action is free almost
everywhere iff $\mathcal{O}_{d} \not= \emptyset$. The``only if"-part
is clear since the dimension of every free orbit equals the group dimension. For the ``if"-part, note that if $\mathcal{O}_d$ is nonempty, then it is
Zariski-open, and its complement has measure zero. Furthermore, the
action of $H$ on $\mathcal{O}_d$ is free: Since group dimension and
orbit dimension coincide, the fixed groups associated to
$\mathcal{O}_d$ are discrete. But they are also connected, hence
trivial.

Finally, we endow the orbit space with a pseudo-image of Lebesgue
measure.

\begin{thm} \label{thm:main}
Let $H$ be a $d$-dimensional connected, admissible abelian matrix
group satisfying ${\rm spec}(h) \subset \mathbb{R}^+$ for all $h \in
H$. If $\pi$ is integrable, there exists a compact conull set $C
\subset \mathcal{O}_d/H^T$.
\end{thm}

The necessary compactness condition simplifies considerably, if more
is known about the orbit space.
\begin{cor} \label{cor:main_Hd}
Let $H$ be a $d$-dimensional connected, admissible abelian matrix
group satisfying ${\rm spec}(h) \subset \mathbb{R}^+$ for all $h \in
H$. Assume in addition that $\mathcal{O}_d/H^T$ is Hausdorff. Then,
if $\pi$ is integrable, $\mathcal{O}_d/H^T$ is compact.
\end{cor}

\begin{prf}
If $\mathcal{O}_d/H^T$ is Hausdorff, then every compact open subset
$C$ is closed as well, and thus it is a union of connected components of $\mathcal{O}_d/H^T$. All
connected components of $\mathcal{O}_d$ are open, thus of positive
Lebesgue measure. Hence all compact components of
$\mathcal{O}_d/H^T$ have positive measure, as well. A compact conull
subset $C$, if it exists, therefore must be all of $\mathcal{O}_d/H^T$.
\end{prf}

A subset $C \subset \mathbb{R}^d$ is
called a {\em topological section} for an $H^T$-invariant subset $U \subset\mathbb{R}^d$
if the map $H \times C \to U,
(h,c) \mapsto h^T c$ is a homeomorphism. Note that the existence of a topological
section implies that $H$ acts freely on $U$. Furthermore, the natural bijection $C \to U/H^T$ is
a homeomorphism as well, implying in particular that the quotient space is Hausdorff.

\begin{cor} \label{cor:int_tcs}
Let $H$ be a $d$-dimensional connected, admissible abelian matrix
group satisfying ${\rm spec}(h) \subset \mathbb{R}^+$ for all $h \in
H$. If there exists a topological section $C \subset \mathcal{O}_d$ for the dual action,
then $\pi$ is integrable iff $C$ is compact.
\end{cor}
\begin{prf}
If $C$ is a topological section, then the canonical bijection $C \to
\mathcal{O}_d/H^T$ is a homeomorphism. In particular, the latter
space is Hausdorff, and it is compact iff $C$ is compact. Thus
one direction is immediate from the previous corollary. The other
direction follows from \cite[Proposition 3.1]{KT2}.
\end{prf}


\section{Topological properties of dual orbit spaces and integrability}

In this subsection we collect some general observations concerning dual
orbit spaces.  We first recall an observation from \cite{KT2}: 
\begin{lemma} \label{lem:comp_subset_os}
Let $U \subset \mathbb{R}^n$ be arbitrary. Then $U/H^T$ is compact iff
there exists a compact set $C \subset U$ with $H^T C= U$.
\end{lemma}
\begin{prf}
Compare the proof of \cite[Corollary 4.2]{KT2}.
\end{prf}

We further note that, as a quotient space of a second
countable space, $\mathbb{R}^n/H^T$ is second countable as well, and 
so are all its subspaces. It follows that a subset is compact iff it is
sequentially compact. Our discussion of concrete cases in the following
sections will mostly rely on this observation. 

Our next aim is to generalize a known sufficient condition for the
existence of integrable admissible vectors. It has been demonstrated in  \cite[Proposition 3.1]{KT2} how to 
construct, given an open conull subset $U$ having a topological section with compact closure inside $U$, 
admissible vectors that give rise to an integrable projection. Our aim is to replace the topological section 
by a weaker notion, a {\em topological quasi-section}. For this purpose, we borrow some notation introduced
by Richard Palais \cite{Pal}:
Given $Y, Z \subset \mathbb{R}^n$,
we let
\[
 ((Y,Z)) = \{ h \in H : h^T Y \cap Z \not= \emptyset \}~.
\]
\begin{defn}
Let $U \subset \mathbb{R}^n$ be open and $H^T$-invariant. 
 A subset $C \subset U$ is called {\em
topological quasi-section} (for $U$) if it is open, with $H^T C=U$, and such that $((C,C))$ is relatively compact.
\end{defn}

The following lemma is clear, but useful to note. 
\begin{lemma} \label{lem:ts_vs_tqs}
Let $U \subset \mathbb{R}^n$ be open and $H^T$-invariant.
Assume that $C \subset U$ is a topological section.
Let $V \subset H$ be a relatively compact open neighborhood of the identity. Then $VC \subset U$ is
a topological quasi-section.
\end{lemma}

We already noted that the existence of a topological section implies that the action is free. 
In the case of a compact quasi-section, a weakened version of this observation applies
\begin{lemma}
 Let $U \subset \mathbb{R}^n$ be open and $H^T$-invariant. If $U$ admits a topological quasi-section, then $H_\xi$ is compact for all $\xi \in U$. 
\end{lemma}
\begin{prf}
 If $C$ is a topological quasi-section, it is sufficient to prove compactness of $H_\xi$, for all $\xi \in C$. 
 But this follows from the fact that $((C,C))$ contains the stabilizer. 
\end{prf}

\begin{lemma} \label{lem:quasi_sect}
Let $U \subset \mathbb{R}^n$ be open and $H^T$-invariant.
If there exists a topological
quasi-section in $U$, then for all compact sets $C_1,C_2 \subset U$,
$((C_1,C_2))$ is relatively compact.
\end{lemma}

\begin{prf}
Denote the quasi-section by $C$. Then, by compactness of the $C_i$,
we have $C_i \subset F_i C$, for suitable finite $F_i \subset H$
($i=1,2$). A straightforward computation shows that $((C_1,C_2))
\subset F_2 ((C,C)) F_1^{-1}$, which is relatively compact.
\end{prf}

We are particularly interested in compact orbit spaces. In this setting, the notion
of a quasi-section becomes particularly useful. In order to properly appreciate the
usefulness of the following quite simple result, recall that the computation of cross-sections
can be quite cumbersome. What is worse, the failure to produce a {\em topological} section
usually does not provide any clue regarding the existence of such a section.

By contrast, as the following proposition shows, there is no ``clever" way of picking quasi-sections,
when the orbit space is compact: Either any relatively compact open set meeting every orbit will do, or none.
Since topological sections give rise to quasi-sections, this provides a useful necessary condition for existence
of topological sections.

\begin{prop} \label{prop:ex_ts}
Let $U \subset \mathbb{R}^n$ be open and $H^T$-invariant.
Assume that $U/H^T$ is compact, and let $C \subset U$ denote a relatively compact open set with $H^TC = U$.
Then $U$ has a topological quasi-section iff $C$ is a topological quasi-section, i.e., iff $((C,C))$ is relatively compact.

In particular, if $((C,C))$ is not relatively compact, then $U$ does not have a topological section.
\end{prop}

\begin{prf}
The ``if"-direction is clear.
The other direction follows from Lemma \ref{lem:quasi_sect}.

The statement concerning topological sections now follows from Lemma \ref{lem:ts_vs_tqs}.
\end{prf}

The following result provides the extension of \cite[Proposition 3.1]{KT2}.
\begin{prop} \label{prop:qs_impl_int}
Let $H< {\rm GL}(n,\mathbb{R})$ be a closed matrix group. Suppose
there exists an open $H^T$-invariant set $U \subset \mathbb{R}^n$
that admits a relatively compact topological quasi-section for the
dual action. Then the restriction $\pi_U$ of the quasi-regular
representation to the invariant closed subspace $\mathcal{H}_U = \{
f \in {\rm L}^2(G): \widehat{f} \cdot 1_U = \widehat{f} \}$ is
integrable.
\end{prop}
\begin{prf}
We adapt the proof from \cite{KT2}. Let $C \subset U$
denote the relatively compact quasi-section. Pick a relatively compact open set
$W$ with $\overline{C} \subset W \subset U$. Let $\phi \in C_c^\infty(U)$ with
$1_{\overline{C} } \le \phi \le 1_{W}$. For $\xi \in
U$, define
\[
\sigma(\xi) = \int_{H} |\phi(h^T \xi)|^2 dh~.
\] Then $\sigma$ is a well-defined smooth function on $U$: Let $\xi \in U$ be arbitrary, and let $U' \subset U$
denote a compact neighborhood of $\xi$.
Then for all $\xi' \in U'$,
\[
\sigma(\xi') = \int_{\overline{((U',\overline{W}))}} |\phi(h^T \xi')|^2 dh~.
\] and Lemma \ref{lem:quasi_sect} ensures that the integration domain is compact.
In particular, $\sigma$ is well-defined, and compactness of the integration domain allows to
interchange differentiation and integration, which shows that $\sigma$ is
smooth.

Furthermore, for every $\xi \in U$, the map $h \mapsto \phi(h^T
\xi)$ is a positive continuous function, and not identically zero: There exists $h \in H$ with $h^T \xi \in C$, and thus
$|\phi(h^T \xi)|^2 \ge 1$ by choice of $\phi$.
Hence $\sigma(\xi)>0$.

But then $\widehat{g} = \phi/|\sigma|^{1/2}$ is a smooth, compactly
supported function on $U$, which we extend trivially to all of
$\mathbb{R}^n$. We let $g \in {\rm L}^2(\mathbb{R}^n)$ denotes its
inverse Fourier transform; by construction, $g$ is in fact a
Schwartz function. Letting $\widehat{g}_h : \xi \mapsto
\widehat{g}(h^T \xi)$, we find that $\widehat{g} \widehat{g}_h = 0$ whenever
$h \not\in C_{W,W}$, and Lemma \ref{lem:quasi_sect} implies that
this set is relatively compact.

Using the properties of $\widehat{g}$ collected thus far, the proof
of \cite[Proposition 3.1]{KT2} is seen to go through. Since our 
notation somewhat differs from that in \cite{KT2}, we include a short sketch.
We first rewrite the wavelet
transform as
\[
V_g g(x,h) =  |{\rm det}(h)|^{1/2}\left( \widehat{g} \cdot \overline{\widehat{g}_h} \right)^{\vee}(x)~.
\] It follows that we can estimate
\[
\| \Delta_G^{-1/2}V_g g \|_{{\rm L}^1} \le
 \int_{H}  \left\|\left( \widehat{g} \cdot \overline{\widehat{g}_h} \right)^{\vee} \right\|_{{\rm L}^1}
 |\det(h)|^{-1/2} \Delta_G^{-1/2}(h) dh ~.
\]
The vector-valued mapping
\[
h \mapsto \widehat{g} \cdot \overline{\widehat{g}_h} \in C_c^\infty(U)
\] is continuous with respect to the Schwartz topology, implying that the mapping
\[ h \mapsto \| (\widehat{g} \cdot \widehat{g}_h)^\vee \|_{{\rm L}^1}
|\det(h)|^{-1/2} \Delta_G^{-1/2}(h) 
\] is continuous as well, and we noted above that it is also compactly supported. This implies $\Delta_G^{-1/2} V_g g
\in {\rm L}^1(G)$. 
\end{prf}

Note that the proof shows that we may replace the factor $\Delta_G^{-1/2}$ by any weight function on $H$ that is bounded on compact sets. 

Unlike for topological sections, the existence of quasi-sections does not imply that the group action is free.
In particular, we can apply the proposition to the case where $U$ is an open dual orbit with compact
stabilizer. The subrepresentations $\pi_U$ corresponding to such orbits are precisely
the discrete series subrepresentations \cite[Corollary 21]{Fu_Calderon}. Hence we obtain
the following result.
\begin{cor}
Let $H< {\rm GL}(n,\mathbb{R})$ be a closed matrix group. Every
discrete series subrepresentation of the quasi-regular representation $\pi$ is integrable.
\end{cor}

While the notion of quasi-section may seem somewhat ad hoc, a result by Palais shows that the
existence of such a quasi-section naturally connects to well-studied topological properties of the orbit space.
Recall for the following that {\em properness} of the dual action on an $H^T$-invariant subset $U$ means that the inverse image of any compact subset $C \subset U \times U$ under the mapping $H \times U \to U \times U$, $(h,\xi) \mapsto (h^T \xi, \xi)$ is compact. 
\begin{prop} \label{prop:proper}
 Let $U \subset \mathbb{R}^n$ be open and $H^T$-invariant. Then the following statements are equivalent:
 \begin{enumerate}
  \item[(a)] There exists a topological quasi-section $C \subset U$ such that $\overline{C}$ is contained in $U$ and compact.
  \item[(b)] The orbit space $U/H^T$ is compact and the dual action of $H$ on $U$ is proper. 
 \end{enumerate}
 If any of these conditions hold, $U/H^T$ is Hausdorff and metrizable. 
\end{prop}
\begin{prf} First assume that $(a)$ holds.  Then $U/H^T$ is clearly compact. 
 We first note that by Lemma \ref{lem:quasi_sect}, we have that $((C_1,C_2))$ is relatively compact,  for all open
 relatively compact subsets $C_1,C_2$ of $U$. In the parlance of \cite{Pal}, this means that $C_1,C_2$ are {\em relatively thin}, and since this holds for all open, relatively compact neighborhoods of arbitrary points, it follows that each point in $U$ has a {\em small neighborhood} (again in the terminology of \cite{Pal}). Now implication $(1) \Rightarrow (3)$ of \cite[Theorem 1.2.9]{Pal} yields that the dual action on $U$ is proper, and we have concluded $(b)$. 
 
 Conversely, assume that $(b)$ holds. Pick a compact set $C \subset U$ with $H^T C = U$. Pick $W \supset C$ open, with 
 closure compactly contained in $U$. Then implication (3) $\Rightarrow$ (4) of \cite[Theorem 1.2.9]{Pal} implies that $((W,W))$ is relatively compact, i.e., $W$ is a topological quasi-section. 
 
 The Hausdorff property of $U/H^T$ follows by \cite[Theorem 1.2.9]{Pal}, (3) $\Rightarrow$ (2), and we already noted that $U/H^T$ is second countable. Metrizability now follows from Urysohn's metrization theorem. 
\end{prf}

\begin{rem}
 The results of this section provide a partial converse to Theorem \ref{thm:main}: Let $H$ be a $d$-dimensional admissible abelian matrix group, and assume that there exists an open conull $H$-invariant set $U \subset \mathcal{O}_d$ such that $U/H^T$ is compact. If in addition the dual action of $H$ on $U$ is proper, then Propositions \ref{prop:qs_impl_int} and \ref{prop:proper} imply that  $\pi$ is integrable.
 
 It is currently unknown whether the gap between this statement and Theorem \ref{thm:main} can actually be closed: Proposition \ref{33} (b) below presents a matrix group $H$ for which the dual action has a conull open compact subset of the orbit space, but the properness condition is violated. For this example, the question of integrability of the associated quasi-regular representation is currently open. 
\end{rem}


\section{Proof of Theorem \ref{thm:main}}

\label{sect:prf_thm_main}

Throughout this section, $H$ denotes an abelian matrix group
fulfilling the conditions of Theorem \ref{thm:main}. In particular, the dual
action of $H$ is free almost everywhere, and thus $d = {\rm dim}(H)
\le n$. Furthermore, $H$ is connected and all elements $h \in H$
have positive spectrum.

Let us first comment on the proof strategy: The chief
representation-theoretic tool employed in \cite{GKT,KT2} was the
Mackey machine. However, while there are useful results providing
topological information (such as continuity of induction), in general the
Mackey machine only yields a full description of the dual {\em as a
set}, and comprehensive information on the topological structure is not easily obtained this way.

For this reason we choose an alternative approach, which is
available because the semidirect products under consideration turn
out to be {\em exponential}. In fact, $G$ is a completely solvable
Lie group (by Lemma \ref{lem:G_cs} below). Thus $\widehat{G}$ is
homeomorphic to the coadjoint orbit space of $G$: Denote by
$\mathfrak{g}$ the Lie algebra of $G$, and by $\mathfrak{g}^*$ its
dual. $G$ acts on $\mathfrak{g}^*$ by the coadjoint representation
${\rm Ad}^*$, and there exists a canonical bijection $\kappa:
\mathfrak{g}^*/{\rm Ad}^*(G) \to \widehat{G}$, the Kirillov-Bernat correspondence. Furthermore, the
bijection is in fact a homeomorphism, when we endow the orbit space
with the quotient topology; see \cite{LeLu} for details. This
observation allows one to employ the methods developed by Currey and
Penney to derive sufficient information on the coadjoint orbit space
for the proof of Theorem \ref{thm:main}. A further important ingredient of the proof
is the close relationship between dual and coadjoint orbits.

\subsection{Adjoint and coadjoint representation of $G$}

We identify elements of the Lie algebra $\mathfrak{g}$ with pairs
$(v,X)^T \in \mathbb{R}^n \times \mathfrak{h}$.
Accordingly, $\mathfrak{g}^*$ is canonically isomorphic to
$\mathbb{R}^n \times \mathfrak{h}^*$; i.e. we have
\[
\langle (\xi,Y^*), (x,X) \rangle = \langle \xi, x \rangle + \langle
Y^*, X \rangle~.
 \] In order to use block matrix calculus, we will write elements of the
product spaces as column vectors, i.e. $\left( \begin{array}{c} w
\\ X
\end{array} \right) \in \mathbb{R}^n \times \mathfrak{h}$,$\left(
\begin{array}{c} \xi \\ X^* \end{array} \right) \in \mathbb{R}^n \times
\mathfrak{h}^\ast$, etc.

The following formulae are conveniently verified by realizing that
the mapping $(x,h) \mapsto \left( \begin{array}{cc} h & x \\ 0 & 1
\end{array} \right)$ is an isomorphism onto a closed group of
$(n+1)\times(n+1)$-matrices. Let ${\rm Ad}: G \to {\rm
GL}(\mathfrak{g})$ and ${\rm Ad}^*: G \to {\rm GL}(\mathfrak{g}^*)$
denote the adjoint and coadjoint representations of the group. Let
${\rm ad}$ and ${\rm ad}^*$ denote their Lie algebra counterparts.

We then obtain the following formulae, using that the adjoint
representation of $H$ is trivial:
\begin{equation} \label{eqn:Ad_rep}
{\rm Ad}(x,h) \left( \begin{array}{c} w \\ Y \end{array} \right)  =
( -{\rm Ad}(h)(Y) \cdot x + h w, {\rm Ad}(h) Y)^T = \left(
\begin{array}{cc} h & \varphi_{x} \\ 0 & 1
\end{array} \right) \cdot ~\left( \begin{array}{c} w \\ Y \end{array} \right),
\end{equation} with $\varphi_{x}: \mathfrak{h} \to
\mathbb{R}^n, Y \mapsto -Y \cdot x$. For the coadjoint
representation, we obtain
\begin{equation} \label{eqn:Ad*_rep}
{\rm Ad}^*(x,h)  \left( \begin{array}{c} \xi \\ Y^* \end{array}
\right)  = \left( \begin{array}{cc} h^{-T} & 0 \\
\varphi^*_{-h^{-1} x} & 1
\end{array} \right) \cdot \left( \begin{array}{c} \xi \\ Y^* \end{array}
\right) ~,
\end{equation}
with $\varphi^{*}_y : \mathbb{R}^n \to \mathfrak{h}^*$ the map dual
to $\varphi_y$.

We next turn to the Lie algebra representations. Here we obtain
\begin{equation} \label{eqn:ad_rep}
{\rm ad}(x,Z) \left( \begin{array}{c} w \\ Y \end{array} \right)  =
\left(
\begin{array}{cc} Z & \varphi_{x} \\ 0 & 0
\end{array} \right) \cdot ~\left( \begin{array}{c} w \\ Y \end{array} \right)~,
\end{equation}
and finally
\begin{equation} \label{eqn:ad*_rep} {\rm ad}^*(x,Z)  \left( \begin{array}{c} \xi \\ Y^*
\end{array}
\right)  = \left( \begin{array}{cc} -Z^T & 0 \\
\varphi_{-x}^* & 0
\end{array} \right) \cdot \left( \begin{array}{c} \xi \\ Y^* \end{array}
\right) ~.
\end{equation}
From these formulae, we can derive a number of simple but important
observations.
\begin{lemma} \label{lem:G_cs}
$G$ is completely solvable.
\end{lemma}
\begin{prf}
By Lemma \ref{eqn:ad_rep}, we find ${\rm spec}(ad(x,Z)) \subset \{0
\} \cup {\rm spec}(Z)$. But the assumption that ${\rm
spec}(\mathbb{R} Z) \subset \mathbb{R}^+$ implies that ${\rm
spec}(Z) \subset \mathbb{R}$, and thus ${\rm spec}(ad(x,Z)) \subset
\mathbb{R}$, for all $(x,\mathbb{Z}) \in \mathfrak{g}$. But this is
equivalent to complete solvability of $G$.
\end{prf}

\

Recall that the Kirillov-Bernat correspondence $\kappa$ is characterized as follows. For $(\xi,Y^*) \in \g^*$, suppose that $\p$ is a subalgebra of $\g$ such that (a) $[\p,\p] \subset \ker (\xi,Y^*)$, and (b) $\p$ is a maximal subalgebra with respect to the property (a). Then one has
\begin{equation}\label{dim}
\dim {\rm Ad}^*(G) (\xi,Y^*) = 2 \dim ( \g / \p).
\end{equation}
For each $(\xi,Y^*) \in\g^*$, there is a subalgebra $\p$ satisfying (a), (b), and also (c) $(\xi,Y^*) + \p^\perp \subset {\rm Ad}^*(G) (\xi,Y^*) $. (Here $\p^\perp = \{ (\zeta, X^*) \in \g^* : \p\subset \ker (\zeta, X^*) \}$. ) Then  $(\xi,Y^*) $ defines a unitary character $\chi  = \chi_{(\xi,Y^*) }$ on  $P= \exp \p$ by 
$$
\chi(\exp (x,X) ) = e^{i\langle (\xi,Y^*),(x,X)\rangle },
$$
and the unitary representation $\pi((\xi,Y^*),\p) := {\rm ind}_P^G(\chi)$ is irreducible. The unitary equi\-valence class $[\pi((\xi,Y^*),\p)] $ of $\pi((\xi,Y^*),\p)$ is independent of the choice of $\p$, as well as independent of the choice of $(\zeta, X^*)\in {\rm Ad}^*(G) (\xi,Y^*)$, and one has $\kappa({\rm Ad}^*(G) (\xi,Y^*))= [\pi((\xi,Y^*),\p)]$. 

The following lemma clarifies the relationship between coadjoint orbits, 
dual orbits, and the dual $\widehat G$, which is particularly transparent for the class of groups that
we study here. For the more general setting, the interested reader is referred
to \cite{AlAnGa}. The formulation of the lemma requires two more pieces of notation:
We let $\mathfrak{h}_\xi = \{ X \in \mathfrak{h} : X^T \xi = 0 \}$;
i.e., $\mathfrak{h}_\xi$ is the Lie algebra of the dual stabilizer
$H_\xi$. Its annihilator is given by $\mathfrak{h}_\xi^\bot = \{ Y^*
\in \mathfrak{h}^* : \langle Y^*, \mathfrak{h}_\xi \rangle = \{ 0 \}
\}$.
\begin{lemma} \label{main_lemma}
\begin{enumerate} 
\item[(a)] For all $(\xi,Y^*) \in \mathfrak{g}^*$ we have
\begin{equation} \label{eqn:coad_vs_dual_orbit}
 {\rm Ad}^*(G) (\xi,Y^*) = \left( H^T \xi \right) \times \left(Y^* +
 \mathfrak{h}_\xi^\bot \right) 
\end{equation}
\item[(b)] The projection map $P: \mathfrak{g}^* \to \mathbb{R}^n$
intertwines the coadjoint and dual actions. It therefore induces a
continuous map $\overline{P}: \mathfrak{g}^*/{\rm Ad}^*(G) \to
\mathbb{R}^n/H^T$.
\item[(c)] Let $\mathcal{O}^*_i \subset \mathfrak{g}^*$ denote the union of orbits of dimensions
$i$. Then $\overline{P}: \mathcal{O}^*_{2d}/{\rm Ad}^*(G) \to
\mathcal{O}_d/H^T$ is a homeomorphism.
\item[(d)] Let $m : \mathcal O_d / H^T \rightarrow \widehat G$ be the Mackey injection defined by $m(H^T\xi) =  [ {\rm ind}_{\RR^n}^G(\chi_\xi)]$. Then the restriction of  $m \circ \overline P$ to $\mathcal O^*_{2d} /{\rm Ad}^*(G)$ coincides with the restriction of  $ \kappa$ to $\mathcal O^*_{2d} /{\rm Ad}^*(G)$
\end{enumerate}
\end{lemma}

\begin{prf} To verify  (\ref{eqn:coad_vs_dual_orbit}), put $\mathfrak p_{(\xi ,Y^*)} = \mathbb R^n \times \mathfrak h_\xi$; using (\ref{eqn:Ad*_rep}) 
we see that  $\mathfrak p_{(\xi ,Y^*)}$ is a subalgebra satisfying 
conditions (a), (b), and (c) with respect to $(\xi,Y^*)$. Hence 
$$\dim  {\rm Ad}^*(G) (\xi,Y^*)  =2 \dim  \mathfrak g  / \mathfrak p_\xi 
= 2 \dim \mathfrak h / \mathfrak h_\xi = 2 \dim H^T\xi.
$$
 A direct calculation again using 
(\ref{eqn:Ad*_rep})  shows that 
$$
 {\rm Ad}^*(G) (\xi,Y^*) \subset \left( H^T \xi \right) \times \left(Y^* +
 \mathfrak{h}_\xi^\bot \right),
 $$
 and it is clear that $\dim \left( H^T \xi \right) \times \left(Y^* +
 \mathfrak{h}_\xi^\bot \right) = 2 \dim H^T \xi$, so  (\ref{eqn:coad_vs_dual_orbit}) follows.

For part (c), we have seen that
the coadjoint orbit dimension of $(\xi, Y^*)$ is twice the dual
orbit dimension of $\xi$. In particular, $(\xi,Y^*) \in
\mathcal{O}_{2d}^*$ iff $\xi \in \mathcal{O}_d$. Furthermore,  $H$
acts freely on $\mathcal{O}_d$, which for $\xi \in \mathcal O_d$ implies
$\mathfrak{h}_\xi^\bot = \mathfrak{h}^*$, and thus ${\rm
Ad}^*(G)(\xi,Y^*) = (H^T \xi) \times \mathfrak{h}^*$. Hence
$\overline{P}$ is bijective on $\mathcal{O}_{2d}^*/{\rm Ad}^*(G)$,
and thus a homeomorphism.

Finally, the above shows that for $(\xi,Y^*) \in \mathcal O^*_{2d}$, the subalgebra $\p = \RR^n$ satisfies conditions (a), (b) and (c) for the Kirillov-Bernat correspondence. 
The characters on $\RR^n$ defined by $(\xi,Y^*)$ and $\xi$ are obviously the same, 
giving the equality of the induced representations 
$\pi((\xi,Y^*),\p) = {\rm ind}_{\RR^n}^G (\chi_\xi)$ so
$$
\kappa( {\rm Ad}^*(G) (\xi,Y^*) ) =[\pi((\xi,Y^*),\p)]=  [ {\rm ind}_{\RR^n}^G (\chi_\xi)] = m(H\cdot \xi).
$$
 \end{prf}
 
\subsection{A layering of $\mathfrak{g}^*$}

 The reduction to dual orbits of maximal dimension will rely on a
corresponding result for the space of coadjoint orbits. The
following lemma contains the vital parts of the proof. For the most
part, it is a rephrasing of the main result of \cite{CP}, adapted to
the our specific needs.
\begin{lemma} \label{lem:layering}
There exists a decomposition of $\mathfrak{g}^*$ into ${\rm
Ad}^*(G)$-invariant, disjoint sets $\left( \Omega_j^*
\right)_{j=1,\ldots,p}$ with the following properties:
\begin{enumerate}
\item[(i)] For all $1 \le k \le p$: $\bigcup_{j \le k} \Omega_j^*$
is open.
\item[(ii)] For all $1 \le j \le p$: All coadjoint orbits in
$\Omega_j^*$ have the same dimension $d_j$. The sequence
$d_1,\ldots,d_p$ is nonincreasing.
\item[(iii)] For all $1 \le j \le p$: $\Omega_j^*/{\rm Ad}^*(G)$
is Hausdorff. \item[(iv)] For all $1 \le j \le p$ we have that
$\Omega_j^* = P^{-1}(P(\Omega_j^*))$.
\item[(v)] Whenever $d_j < 2d$, all connected
components of $\Omega_j^*$ are noncompact.
\end{enumerate}
\end{lemma}
\begin{prf}
The proof relies on the explicit construction of cross-sections for
the coadjoint orbits of $G$, as described in \cite{CP}. The starting
point for this construction is a Jordan-H\"older basis of
$\mathfrak{g}$. In our setting, this can be obtained with
particular ease: By \cite[Theorem 4]{BCFM}, there exists a basis of
$\mathbb{R}^n$ with respect to which all elements of $H$ are upper
triangular. We denote this basis by $X_1,\ldots,X_n$, and complement
it by a basis $X_{n+1},\ldots,X_{n+d}$ of $\mathfrak{h}$. It is then
straightforward to check that
\[
\mathfrak{g}_j = {\rm span}(X_i: i \le j)~,
\] defines an increasing sequence of ideals $\mathfrak{g}_j$ in $\mathfrak{g}$. Hence
$X_1,\ldots,X_{d+n}$ is the desired Jordan-H\"older basis.

Consequently, the dual basis $X_1^*, \ldots,X_{d+n}^* \subset
\mathfrak{g}^*$ is a Jordan-H\"older basis for the dual, with the
usual, implicit convention that this time, the subspaces
\[
V_j = {\rm span}(X^*_i : i > j)
\] are submodules. It follows that $G$ acts on $\mathfrak{g}^*/V_j$,
and the construction of the layering will be obtained from data
associated to the actions of $G$ on the various quotient spaces.

Given $(\xi,Y^*)^T \in \mathfrak{g}^*$ and $1 \le j \le n+d$, we let
\[
L_j(\xi,Y^*) = \{ (w,X) \in \mathfrak{g} : {\rm ad}^*(w,X)
(\xi,Y^*)^T \in V_j \}~.
\]
By definition, $L_j(\xi,Y^*)$ is therefore the Lie algebra of the
fixed group of $(\xi, Y^*)^T + V_j$ with respect to the action of
$G$.
%
%
%
%
%

Now, given $L_j(\xi,Y^*)$, for $j=1,\ldots,n+d$, R. Penney and the first author describe
 in \cite{CP} how to associate an index set 
$\mathfrak{j}(\xi,Y^*)$ and an ordering scheme $\alpha(\xi,Y^*)$
for $\mathfrak{j}(\xi,Y^*)$,
in such a way that the
layering is obtained by picking suitable tuples
$\mathfrak{j},\alpha$ and letting
\begin{equation} \label{eqn:layering}
\Omega_{\mathfrak{j},\alpha} = \{ (\xi,Y^*):
\mathfrak{j}(\xi,Y^*) = \mathfrak{j}, \alpha(\xi,Y^*) =
\alpha \}~. \end{equation} Here
$\mathfrak{j}(\xi,Y^*)$ is the set of {\em jump indices} for $(\xi,Y^*)$, indicating
when the dimension of ${\rm Ad}^*(G)(\xi,Y^*)+V_{j+1}$ exceeds that
of ${\rm Ad}^*(G)(\xi,Y^*)+V_{j}$. Hence, the cardinality of the
jump set is equal to orbit dimension.

In addition, one can number the nonempty layers as
$\Omega_1,\ldots,\Omega_p$ in such a way that property (i) is
guaranteed (see \cite[Main Theorem (c)]{CP}). In each layer, the
dimensions in the coadjoint orbits is constant (see \cite[Main
Theorem (b)]{CP}). The property that the orbit dimension is
nondecreasing follows since the layering
is a refinement of the layering described in \cite{Pe}, as pointed
out on \cite[p. 311]{CP}. One can check that the coarser
layering has this property by consulting the definition of the total
ordering given directly after \cite[Corollary 1.1.2]{Pe}.

Furthermore, there exists for each layer $\Omega_i$ a cross-section
$\Sigma_i \subset \Omega_i$  and a diffeomorphism $\Theta_i :
\Sigma_i \times W_i \to \Omega$ (with a suitably chosen vector space
$W_i$) such that for each $(\xi,Y^*) \in \Sigma$ we have
$\Theta_i((\xi,Y^*),W_i) = {\rm Ad}^*(G) (\xi,Y^*)$ \cite[Main
Theorem (f)]{CP}. But this easily implies that $\Omega_i/{\rm
Ad}^*(G)$ is homeomorphic to $\Sigma_i$, in particular, it is
Hausdorff. This takes care of $(iii)$.

For the proof of part (iv) we observe that, since ${\rm ad}^*(x,Z)
(\xi^*, Y^*) = {\rm ad}^*(x,Z)(\xi^*,0)$ (see (\ref{eqn:ad*_rep})),
it follows that $L_j(\xi,Y^*) = L_j(\xi,0)$, for all $1 \le j \le
n+d$. But this implies $\mathfrak{j}(\xi,Y^*) =
\mathfrak{j}(\xi,0)$, and the same for $\alpha$, and (iv) follows.

Now (v) is an easy consequence of (iv) and of
(\ref{eqn:coad_vs_dual_orbit}): Assume that $(\xi,Y^*)^T \in
\Omega_j$, and $d_j < 2d$. Then $d_j = 2k$ with $k< d$, and thus
${\rm codim}(\mathfrak{h}_\xi^\bot)\ge 1$. Let $V \subset
\mathfrak{h}^*$ denote a complement of $\mathfrak{h}_\xi^\bot$. Then
the map
\[
V \ni Z^* \mapsto (H^T \xi) \times \left( Z^* +
\mathfrak{h}_\xi^\bot \right)
\] defines a homeomorphism onto the closed subset $\overline{P}^{-1}(\{ H^T \xi \})$ of
$\Omega_i^* /{\rm Ad}^*(G)$. This set contains the coadjoint orbit
of $(\xi,Y^*)^T$, and it is connected and noncompact. Hence the
connected component in $\Omega_i^* /{\rm Ad}^*(G)$ containing
that orbit is not compact.
\end{prf}

The properties of the lemma will allow us 
to``peel off" all layers in
the coadjoint orbit space that consist of orbits of less than
maximal dimension.
\begin{thm} \label{thm:main_coadj}
Let $C \subset \mathfrak{g}^*/{\rm Ad}^*(G)$ denote an open compact
set. Then $C \subset \mathcal{O}^*_{2d}$.
\end{thm}
\begin{proof}
Let $\Omega_1^*,\ldots,\Omega_p^*$ denote the layering from Lemma
\ref{lem:layering}. Note that $\Omega_p^*$ consists of the
zero-dimensional orbits. Suppose that $C \subset \mathfrak{g}^*/{\rm
Ad}^*(G)$ is compact and open. Then $C_p = C \cap \Omega_p^*/{\rm
Ad}^*(G)$ is a compact subset of $\Omega_p^*/{\rm Ad}^*(G)$, which
is Hausdorff by part (iii) of Lemma \ref{lem:layering}. Furthermore,
it is open, hence it contains every connected components of
$\Omega_p^*/{\rm Ad}^*(G)$ it intersects. But these are noncompact,
by Lemma \ref{lem:layering}.(v). Thus $C_p \cap \Omega_p^*/{\rm
Ad}^*(G)$ is empty.

If $d_{p-1}<d$, we can repeat the argument, using that
$\Omega_{p-1}^* \subset \bigcup_{j \le n-1} \Omega_j^*$ is closed,
to obtain that $C \cap \Omega_{p-1}^*/{\rm Ad}^*(G) = \emptyset$. Thus
one concludes iteratively $C \cap \Omega_{j}^*/{\rm Ad}^*(G) =
\emptyset$, until $d_j = d$.
\end{proof}

%
%
%
%
%
%
%
%

\subsection{Proof of Theorem \ref{thm:main}}

Suppose that $\eta \in {\rm L}^2(\RR^n)$ is an admissible vector with $F
= \Delta_G^{-1/2} V_\eta \eta \in {\rm L}^1(G)$. Denote by $\widehat{F}$ the group
Fourier transform. Then ${\rm supp}(\widehat F) \subset \widehat{G}$ is a
compact open set. Let $\kappa: \mathfrak{g}^*/{\rm Ad}^*(G) \to
\widehat{G}$ denote the Kirillov-Bernat correspondence. It is a
homeomorphism \cite{LeLu}, and thus Theorem \ref{thm:main_coadj} implies
that ${\rm supp}(\widehat F) = \kappa(C)$ for some compact open subset $C\subset
\mathcal{O}^*_{2d}/{\rm Ad}^*(G))$. By Lemma \ref{main_lemma}, 
 $\overline{P}(C)$ is a compact open subset of $\mathcal O_d / H^T$. 

Now let $\mathcal H = V_\eta(L^2(\RR^n)) \subset L^2(G)$. 
The measure class on the Borel space $\widehat G$ for the 
quasiregular representation $\pi$ is supported on 
$\eO_{2d}^* / \text{Ad}^*(G) = \eO_d/H^T$, and is the 
push forward of the Plancherel measure on 
$\widehat \R^n$ (see for example \cite[Theorem 7.1]{Lip}). 
Hence the restriction of the regular representation of $G$ to $\mathcal H$
is quasiequivalent with the regular representation.
Since convolution by $F$ is projection of $L^2(G)$  onto $\mathcal H$, then 
the support of $\widehat F$ must be conull. This entails that $C$ is conull, and hence
$\overline{P}(C) \subset \mathbb{R}^n$ must be conull as well. 

%
%
%
%


\section{The $3\times 3$ case.} 

In this section we specialize to the case where $n = 3$, with the aim of proving a characterization. The proposition lists each connected abelian $H< GL(3,\R)$ for which $\R^3 / H^T$ contains a compact  open subset. Of particular interest is case (b) for $\alpha \beta \not= 0$, providing an example of a conull compact open subset of the orbit space where Proposition \ref{prop:qs_impl_int} is not applicable. Even though this problem seems amenable to direct calculations, the question whether an integrable admissible vector exists for this case is currently open. 

\begin{prop} \label{33} Let $\h$ be a commutative Lie subalgebra of $\mathfrak{gl}(3,\RR)$ having dimension $d \ge 2$, and suppose that there is a basis of $\R^3$ for which one of the following cases holds.

\medskip
\noindent
%
%
{\rm (a)} $\h^T = (A,B)_\R$ where 
 $$
A = \left[ \begin{matrix} 1 &\! -\alpha& \\ \alpha&\ 1&\\&&0\end{matrix} \right], \ \ \ \text{and} 
 \ \ \ B = \left[\begin{matrix} 0&& \\ &0& \\&&1\end{matrix}\right].
 $$
 with $\alpha \ne 0$. 

 \medskip
 \noindent
{\rm (b)} $\h^T = (A,B)_\R$ where 
$$
A = \left[ \begin{matrix} 1 && \\ &0&\\&&\alpha\end{matrix} \right] \ \ \ \text{and} \ \ \ 
B = \left[ \begin{matrix} 0 && \\ &1&\\&&\b\end{matrix} \right].
$$
with at least one of $\alpha $ or $\b$ positive. 

\medskip
 \noindent
{\rm (c)} $\h^T = (A,X, Z)_\R$ where 
$$
A = \left[ \begin{matrix} 1 && \\ &1&\\&&1 \end{matrix} \right] , \ \ \ 
X = \left[ \begin{matrix} 0 && \\ 1&0&\\&&0\end{matrix} \right], \ \ \ \text{and} \ \ \ 
Y = \left[ \begin{matrix} 0 && \\ &0&\\1&&0\end{matrix} \right].
$$

 \medskip
 \noindent
{\rm (d)} $\h^T = (A,B, X)_\R$ where 
$$
A = \left[ \begin{matrix} 1 && \\ &1&\\&&0 \end{matrix} \right] , \ \ \ 
B = \left[ \begin{matrix} 0 && \\ &0&\\&&1\end{matrix} \right], \ \ \ \text{and} \ \ \ 
X = \left[ \begin{matrix} 0 && \\ 1&0&\\&&0\end{matrix} \right].
$$

 \medskip
 \noindent
{\rm (e)} $\h^T = (A,B, C)_\R$ where 
$$
A = \left[ \begin{matrix} 1 && \\ &0&\\&&0\end{matrix} \right], \ \ \ 
B = \left[ \begin{matrix} 0 && \\ &1&\\&&0\end{matrix} \right],  \ \ \ \text{and} \ \ \ 
C = \left[ \begin{matrix} 0 && \\ &0&\\&&1\end{matrix} \right].
$$

\medskip
\noindent
Then $\eO_d / H^T$ is compact.  Otherwise, $\eO_d / H^T$ contains no conull, compact open subset.

In case $(b)$, there exists a topological section for $\eO_d/H^T$ iff $\alpha \beta = 0$, whereas for $\alpha \beta \not= 0$ there does not even exist a topological quasi-section. In all remaining cases, $\eO_2/H^T$ admits a topological section. 

If none of the above cases applies, then $\pi$ is not integrable. Conversely, $\pi$ is integrable in all listed cases possessing a topological section. 

\end{prop}

Note that when $d = 3$, $\eO_d$ consists of open orbits, and it is easy to check directly that the groups generated in examples (c), (d), and (e) above generate finitely many open orbits. We elaborate on this in the cases below, but the main focus of the proof will be on the two-dimensional case.

We first decompose $\R^3$ into generalized eigenspaces for $\h$. Let $\l$ be a complex linear form on $\h$, and put 
$$
E_\l = \cap _{A \in \h} \ker_{\C^3}(A - \l(A) I)^2
$$
where each $A$ is extended to an endomorphism of $\C^3$ as usual. If $E_\l \ne \{0\}$, then we say that $\l$ is root for $\h$. Since $\h$ is commutative, there is a finite set $\mathcal R$ of roots such that
$$
\C^3 = \oplus_{\l \in \mathcal R} E_\l.
$$
Accordingly we write $v = \sum_{\l \in \mathcal R} v^{(\l)}$.
If $\l \in \mathcal R$ then $\overline{\l} \in \mathcal R$ also, and for $v \in \R^3$, $v^{(\overline{\l})} = \overline{v^{(\l)}}$. 
We then form a set $\{ \l_1, \dots , \l_p\}$ by choosing exactly one member of each conjugate pair in $\mathcal R$. If $\l$ is real, put $V_j = E_{\l_j} \cap \R^3$, and otherwise put $V_j = E_{\l_j}$. Then the map $\iota : \R^3 \rightarrow V_1 \oplus \cdots \oplus V_p$ defined by $\iota(v) = \sum_{j=1}^p v^{(\l_j)}$ is an $\R$-linear isomorphism. 

The preceding generalized eigenspace decomposition can be carried out for any commutative Lie algebra $\h $ in $\mathfrak{gl}(n,\R)$. We now make some observations particular to the situation at hand, where $n = 3$. First, note that if $\l_j$ is not real, then $V_j = \C \simeq \R^2$ and $p  = 2$. Second, observe that if there is a non-zero nilpotent element in $\h$, then $p \le 2$ and all roots are real. 
Let $\n = \cap_{j=1}^p \ker \l_j$ denote the subspace of nilpotent elements of $\h$. In light of these observations we itemize several cases.

 \medskip
 \noindent
 {\bf Case 0:} $p = 1$ and $\l_1 = 0$. Here $\h = \n$ is a Lie subalgebra of the Heisenberg Lie algebra, and our assumption that $\h$ is commutative then requires that $d = 2$.

 \medskip
 \noindent
 {\bf Case 1:} $p = 1$ and $\l_1 \ne 0$. Here the dimension of $\n$ is one or two, and $\l_1$ is real. Suppose that $\dim \n = 2$; then for any basis of $\R^3$ there must be $X \in \n$  for which $x_{21} \ne 0$, and then the assumption that $\h$ is commutative implies that for all $Y \in \n$, $y_{32} = 0$. This is example (d). It is easily seen that in this case $\eO_3$ is the union of two open orbits.

 \medskip
 \noindent
 {\bf Case 2:} $p = 2$ and $\l_1$ and $\l_2$ are linearly dependent over $\R$. Here one of the generalized eigenspaces has dimension one, and it follows that the dimension of $\n$ is one, so $d = 2$. 
 
 \medskip
 \noindent
 {\bf Case 3:} $p = 2$ and $\l_1$ and $\l_2$ are linearly independent over $\R$. In this case, if $\n \ne \{0\}$, then both roots are real, and we are necessarily in example (d) above, whence $\eO_3$ is the union of four open orbits. Otherwise there can be a complex root.

 \medskip
 \noindent
 {\bf Case 4:} $p = 3$. Here we must have $\n = \{0\}$, and the dimension of  the real span of the roots is $d$. If $d = 3$ then we clearly obtain example (e), where $\eO_3$ is a union of eight open orbits. Otherwise, with $\l_1$ and $\l_2$ linearly independent, the question of compactness of $\eO_2/H^T$ will depend upon how $\l_3$ is written as a linear combination of $\l_1$ and $\l_2$.

\

The following lemma allows us to assume that all roots are non-vanishing, and in particular, to eliminate Case 0.

\

\begin{lemma} \label{unipotent} Suppose that $\l = 0$ for some $\l \in \mathcal R$. Then $\R^3 / H^T$ contains no non-empty compact open subset. 

\end{lemma}

\begin{proof} We have a generalized eigenspace $V_0$ in which the action of $H^T$ is unipotent. Choose a basis $e_1, \dots , e_{n_0}$ for $V_0$ for which every element of $H^T$ has a lower triangular matrix, and let $U$ be the span of $e_2, \dots e_{n_1}$ together with any other generalized eigenspaces  (if $p = 1$ and $n_1 = 1$ then $U = \{0\}$). Then $U$ is $H^T$-invariant.  Let $\pi :\R^3 \rightarrow \R^3 / U = \R e_1$ be the canonical map. Then $\pi$ is continuous,  open, and $H^T$-invariant, and so gives a well-defined, continuous open map $\overline{\pi} : \R^3 / H^T \rightarrow \R e_1$.  If $C$ is a compact open subset of $\R^3 / H^T$, then $\overline{\pi}(C)$ is compact and open in $\Re_1$, hence empty. Thus $C$ is empty.

\end{proof}

Next we prove two lemmas  for Case 1. When an ordered basis $\{e_1, e_2, e_3\}$  for $\R^3$ is chosen, let $p_i : \R^3 \rightarrow \R$ be the linear (coordinate) function dual to $e_i$. 

\

\begin{lemma} Suppose that $d = 2$ and we are in Case 1, and write $\h^T = (A,X)_\R$ where $\l_1(A) = 1$ and $\l_1(X) = 0$. Choose a basis for $\R^3$ for which $\h^T$ is lower triangular, and write 
$$X = \left[ \begin{matrix} 0&&\\x_{21}&0&\\x_{31}&x_{32}&0\end{matrix}\right]
$$
Put
$$
\Omega_2 = \{ v  : p_2(Xv) \ne 0\}
$$
and 
$$
\Omega_3 = \{ v : p_2(Xv) = 0, p_3(Xv) \ne 0\}.
$$
Then the sets $\Omega_2$ and $\Omega_3$ are $H^T$-invariant and $\eO_2 = \Omega_2 \cup \Omega_3$. When $\Omega_b$ is nonempty, then 
$$
\Sigma_b = \{ v \in \O_b : p_b(v) = 0 \text{ and }\  p_b(Xv) = \pm 1\}
$$
is a topological section for the $H^T$-orbits in $\Omega_b$.

\end{lemma}


\begin{proof}  We have $\Omega_2 \ne \emptyset$ if and only if $x_{21} \ne 0$.  In this case $p_2(Xv) = x_{21}v_1 $, and so $\Omega_2 = \{ v : v_1 \ne 0\}$. Thus $\Omega_2$ is $H$-invariant with 
$$
p_2(Xe^{sA} v) = e^{s\l_1(A)} p_2(Xv).
$$
Now suppose that $v_1 = 0$. If $p_3(Xv) \ne 0$ for some $v$ then $p = 1$ (i.e., there is only one generalized Jordan block ) and $p_3(e^{sA} Xv) = e^{s\l_1(A)} p_3(Xv)$ for all $v$. 

Observe that if $\zO_b$ is nonempty, then the restriction of the map $\phi : \Omega_b \rightarrow \R^* \times \R$ defined by 
$$
\phi(v) = (p_b(Xv), v_b)
$$
to each $H$-orbit has rank 2. Thus $\zO_b \subset \eO_2$. If $v \in \eO_2$ then the action of $e^{\R X}$ is non-trivial and so $v \in \zO_b, b = 2$ or $3$.

To see that $\Sigma_b$ is a section, let $v\in \zO_b$. Note that by definition of $b$, $p_b(X^2v) = 0$, so for any $t \in \R$, 
$$
p_b(e^{tX}  v) = p_b(v) + t p_b(Xv).
$$
Put $s = - \ln |p_b(Xv)|$,  and 
$$
t = -  \frac{p_b(v)}{p_b(X  v)}.
$$
Put $v^* = e^{tX + sA}v$. Then $p_b(X v^*) = p_b(Xe^{sA} v) = e^{s} p_b(Xv) = \pm 1$, and  
$$
p_b(v^*) = p_b( \ v) + t p_b(X \ v) = 0.
$$
Thus $v^*$ belongs to $\Sigma_b$. On the other hand if $v^* \in \Sigma_b$ and $v= e^{tX + sA}v \in \Sigma_b$ also, then $p_b(Xv) = p_b(Xe^{sA} v^*) = e^s p_b(Xv^*) $ implies $s = 0$ and 
$$
0 = p_b(v) = p_b(e^{tX} v) =  t p_b(Xv)
$$
shows that $t = 0$. Thus $\Sigma_b$ meets each orbit at exactly one point. Furthermore, the explicit computations determining $s,t$ and finally $v^*$ from $v$ show that the map $v \mapsto v^*$ is continuous. This implies that the section is topological. 
\end{proof}

\begin{lemma} \label{1layer} Suppose that $d = 2$ and we are in Case 1 and that $\eO_2$ consists of only one layer:  $\eO_2 = \zO_a = \zO$. Then $\eO_2/H^T$ does not contain a non-empty compact open subset. 

\end{lemma}

\begin{proof} Let $\Sigma$ be the topological section  for $\eO_2$ constructed above, let $\sigma : \eO_2/H^T \rightarrow \Sigma$ be the natural homeomorphism, and let $C$ be a compact open subset of $\eO_2 / H^T$. Then $\sigma(C)$ is a compact and open subset of $\Sigma$. Since $\Sigma$ is carries the relative topology from $\R^3$, then $\sigma(C)$ is closed and bounded. Now if $\sigma(K)$ were nonempty, then it would contain a component of $\Sigma$. But the components of $\Sigma$ are unbounded, so this is impossible, and so $\sigma(C)$, and hence $C$, are empty. 
\end{proof}

We now prove the proposition.

\begin{proof}[Proof of Proposition \ref{33}] We proceed on a case by case basis, with the assumption that $d = 2$. 

{\bf Case 0.} We have already established in Lemma \ref{unipotent} that $\mathcal{O}_2/H^T$ does not contain an open compact subset, and since $\mathfrak{h}$ is conjugate to a  Lie subalgebra of the Heisenberg Lie algebra, Theorem \ref{thm:main} applies to show that $\pi$ is not integrable. 

\medskip
\noindent
{\bf Case 1.} First recall that here $\mathfrak{n}$ is nontrivial, and all roots are real. Hence Theorem \ref{thm:main} applies once again, and yields non-integrability whenever $\mathcal{O}_2/H^T$ does not contain nontrivial compact open subsets.  Here $\l$ is real and $ \h = (A,X)_\R$ where $\l(A) = 1$ and $\l(X) = 0$. Choose a basis for $\R^3$ for which $A $ has Jordan form, and for which $X$ is lower triangular. Put $Y = A - \l(A)I$. We distinguish several subcases.

\medskip
\noindent
(a) Suppose that $Y = 0$. If $\eO_2$ consists of only one layer, then confer the preceding lemma. Suppose then that $\eO_2 $ is the union of the non-empty sets $ \zO_2$ and $ \zO_3$. Then $\zO_2 \ne \emptyset $ implies $x_{21} \ne 0$, while $\zO_3 \ne \emptyset$ implies that for some $v$, $v_ 1 = 0$ and $p_3(Xv) \ne 0$, hence $x_{32} \ne 0$. 

Let $U$ be an open $H^T$-invariant subset of $\eO_2$ that is conull for Lebesgue measure. Then $U\cap \zO_2$ is conull and $H^T$-invariant, hence $U \cap \Sigma_2$ is conull for the Plancherel measure class, where $ \Sigma_2 = \{ [\pm 1, 0,v_3 ] : v_3 \in \R\}$. But in this case the Plancherel measure class is that of Lebesgue measure on each component of $\Sigma_2$, so in particular, the set 
$$
E = \{ z \in \R : [1,0,z]^t  \in U\}
$$
is conull for Lebesgue measure. Now choose a sequence $z(k) \in E$ such that $z(k) \rightarrow \infty$ and put $v(k) = [1,0,\text{sign}(x_{21}x_{32})z(k)]^t, k = 1, 2, 3 \dots $. We claim that the sequence $H^Tv(k)$ has no convergent subsequence in $\eO_2 / H^T$. 

Suppose that the claim is false. Let $v(j) = [1,0,\text{sign}(x_{21}x_{32})z(j)]^t$ be a subsequence of $v(k)$ for which $H^Tv(j)$  converges to $H^T v$, for some $v \in \eO_2$. Since $v(k)$ is not convergent in $\Sigma_2$, then $v \notin \Sigma_2$, and $H^T v\in \zO_3/H^T$. Hence we may assume that $v = [0,\pm 1,0]^t \in \Sigma_3$. It follows from the definition of the quotient topology that we have $s_j, t_j \in \R$ such that $e^{s_j A + t_j X}v(j) \rightarrow v$; that is, 
$$
e^{s_j} \ [ \pm 1, x_{21}t_j, \text{sign}(x_{21}x_{32})z(j) + x_{31} t_j + x_{21}x_{32} t_j^2 / 2 ] \longrightarrow [0,\pm1,0]^t.
$$
Thus $e^{s_j} \rightarrow 0$ while $e^{s_j}x_{21} t_j \rightarrow \pm 1$, and 
\begin{equation}\label{coord3}
\ e^{s_j} \bigl( \text{sign}(x_{21}x_{32})z(j) + x_{31} t_j + x_{21}x_{32} t_j^2 / 2\bigr) \longrightarrow 0.
 \end{equation}
Now $e^{s_j} \rightarrow 0$ and  $e^{s_j}x_{21} t_j \rightarrow \pm 1$ implies that $ |t_j| \rightarrow +\infty$, and $e^{s_j } t_j^2 \rightarrow +\infty$. But then 
$$
\text{sign}(x_{21}x_{32}) \ e^{s_j} \bigl( \text{sign}(x_{21}x_{32})z(j) + x_{31} t_j + x_{21}x_{32} t_j^2 / 2 \bigr) \longrightarrow +\infty 
$$
contradicting (\ref{coord3}).  Thus the claim is proved, and $U$ is not compact.

\medskip
\noindent
(b)  Suppose that $Y \ne 0, Y^2 = 0$, that is , that either
$$
A = \left[ \begin{matrix} \l && \\1&\l&\\&&\l\end{matrix} \right] \ \ \text{ or } \ \ 
A = \left[ \begin{matrix} \l && \\&\l&\\&1&\l\end{matrix} \right].
$$
Then $X$ has the form
$$
X =  \left[ \begin{matrix} 0 && \\x_{21}&0&\\x_{31}&&0\end{matrix} \right] \ \ \text{ or } \ \
X =  \left[ \begin{matrix} 0 && \\&0&\\x_{31}&x_{32}&0\end{matrix} \right]  
$$
respectively. If $x_{21} \ne 0$ (in the first case) then $\eO_2 = \zO_2$; otherwise, $\eO_2 = \zO_3$. In any case Lemma \ref{1layer} applies.

\medskip
\noindent
(c) Suppose that $Y^2 \ne 0$, that is, that
$$
A = \left[ \begin{matrix} \l && \\1&\l&\\&1&\l\end{matrix} \right].
$$
Then $X$ has the form
$$
X =  \left[ \begin{matrix} 0 && \\x_{21}&0&\\x_{31}&x_{32}&0\end{matrix} \right]
$$
where $b = x_{21} = x_{32}$. If $b = 0$ then $\eO_2 = \zO_3$ and Lemma \ref{1layer} applies. Suppose then that $b \ne 0$. We have $\eO_2 = \zO_2 \cup \zO_3$ and both are non-empty.

We assume $b = 1$, and write $c = x_{31}$. We compute for any $v \in \R^3$, 
$$
e^{sA + tX} v = e^{\l s} \ [ v_1, v_2 + (s+t) v_1, v_3 + (s+t) v_2 + (ct + (1/2)(s+t)^2)v_1].
$$
The argument now proceeds in a manner similar to that of Case 1(a). Let $U$ be an open $H$-invariant subset of $\eO_2$ that is conull for Lebesgue measure. As before we find that the set 
$$
E = \{ z \in \R : [1,0,z]^t  \in U\}
$$
is conull for Lebesgue measure. Now choose a sequence $z(k) \in E$ such that $z(k) \rightarrow \infty$ and put $v(k) = [1,0,z(k)]^t, k = 1, 2, 3 \dots $. We claim that the sequence $H^Tv(k)$ has no convergent subsequence in $\eO_2 / H^T$. Supposing the claim is false, then as before this leads us to the setup: 
$$
e^{ s_j A + t_j X}v(j) \longrightarrow v
$$
where $v(j)$ is a subsequence of $v(k)$ and where $v  = [0,\pm 1,0]^t \in \Sigma_3$. More explicitly this means that 
$$
e^{\l s_j} \ [ 1, s_j + t_j, z(j) + ct_j + (1/2)(s_j + t_j)^2 ] \longrightarrow [0,\pm 1,0].
$$
Thus $e^{\l s_j} \rightarrow 0$, and $e^{\l s_j} (s_j+ t_j) \rightarrow \pm 1$, and
\begin{equation}\label{coord3again}
e^{\l s_j} \bigl( z(j) + ct_j + (1/2)(s_j + t_j)^2\bigr) \longrightarrow 0.
\end{equation}
The first two relations imply that $e^{\l s_j} (s_j+ t_j)^2 \rightarrow +\infty,$ and moreover (since $e^{\l s_j}s_j \rightarrow 0$) that $e^{\l s_j}t_j \rightarrow \pm 1$. But $ e^{\l s_j} z(j) > 0$ and so $e^{\l s_j} (s_j+ t_j)^2 \rightarrow +\infty$ and $e^{\l s_j}t_j \rightarrow \pm 1$ implies that 
$$
e^{\l s_j} \big( z(j) + ct_j + (1/2)(s_j + t_j)^2\bigr) \longrightarrow +\infty
$$
contradicting (\ref{coord3again}). Therefore the set $U$ is not compact.

\medskip
\noindent
{\bf Case 2: } Here we have two dependent roots $\l_1$ and $\l_2 $, both real. One of $V_1$ or $V_2$ has dimension two; assume it is $V_1$. Then we have a basis for $\R^3$ and $X\in \h$ with matrix form
$$
X = \left[\begin{matrix} 0&& \\ 1&0& \\&&0\end{matrix}\right].
$$

By Lemma \ref{unipotent} we may assume that both $\l_1$ or $\l_2$ are non-vanishing. Hence we can choose a basis for $\R^3$ such that $\h = (A,X)_\R$ where $X$ is as above, and where
$$
A = \left[ \begin{matrix} 1 && \\ &1&\\&&\alpha\end{matrix} \right]
$$
for some $ \alpha \ne 0$. Here it is easily seen that $\eO_2 = \zO_2 = \{ v : v_1 \ne 0\}$, for which we have a global, connected and noncompact topological section $\Sigma_2 = \{ v : v_1 = \pm 1, v_2 = 0, v_3 \in \mathbb{R} \}$. The Plancherel measure, when transferred to the section $\Sigma_2$, is again seen to be equivalent to Lebesgue measure with respect to the third coordinate. Hence, the same reasoning as in the previous cases shows that $\eO_2 / H^T$ cannot contain a conull compact open subset. 

Again, non-integrability follows from Theorem \ref{thm:main}. 

\medskip
\noindent
{\bf Case 3:}  $p = 2$, $\l_1 $ and $  \l_2$ are independent over $\R$. Write  $\h = (A, B )_\R$ where $\l_2(B) = 1$, $\l_2(A) = \l_1 (B) = 0$, and where $\l_1(A) = c \in \C$.

\medskip
\noindent
(a) Suppose that $\Re c \ne 0$. We have $\iota : \R^3 \rightarrow \C \times \R$ defined by 
$\iota([v_1, v_2, v_3]^t =[v_1 + i v_2, v_3]^t$, and via this realization we can write
$$
A = \left[ \begin{matrix} c&\\&0\end{matrix}\right], \ \ \ B  = \left[ \begin{matrix} 0&\\&1\end{matrix}\right]
$$
It is clear that $\eO_2 = \{ v : v_1 + iv_2 \ne 0, v_3 \ne 0\}$ and that a topological section for $\eO_2$ is $\Sigma = \{ v : v_1^2 + v_2^2 = 1, |v_3| = 1\}$. This section is compact so $\eO_2 / H $ is compact. 

\medskip
\noindent
(b)  Suppose that $c = i$. Let $\pi : \eO_2 \rightarrow \C^*$ be the map $\pi(v) = v_1 + iv_2$. Then  $\pi$ is continuous and open and $\pi( e^{sA + tB}v) = e^{is} (v_1 + iv_2)$; hence we have the continuous open mapping $\overline{\pi} : \eO_2 / H \rightarrow \C^* / H = (0,+\infty)$. If $K$ is a compact open subset of $\eO_2 / H$ then $\overline{\pi}(K)$ is a compact open subset of $(0,+\infty)$, hence empty, so $K$ is empty. 

For non-integrability of $\pi$, we observe that the block-diagonal structure of $H$ allows to write 
\[ \mathbb{R}^3 \rtimes H \cong \left( \mathbb{R}^2 \semdir {\rm SO}(2) \right) \times \left( \mathbb{R} \rtimes \mathbb{R}' \right) ~.\]
It follows that the dual space of $G= \mathbb{R}^3 \rtimes H$ is the cartesian product of the unitary duals of the factors.
In particular, projecting a nonempty open compact subset of $\widehat{G}$ onto the first component would yield a nonempty open compact subset in the dual of $\mathbb{R}^2 \semdir {\rm SO}(2)$, contradicting Theorem 1.1 of \cite{GKT}. Hence there are no integrable projections in ${\rm L}^1(G)$.  


\medskip
\noindent
{\bf Case 4:} Here each $\l_j$ is real and the action of $H^T$ is semi-simple. Two of the roots are independent, and we can write $\h = (A, B)_\R$ where 
$$
A = \left[ \begin{matrix} 1 && \\ &0&\\&&\alpha\end{matrix} \right]
$$
and 
$$
B = \left[ \begin{matrix} 0 && \\ &1&\\&&\b\end{matrix} \right].
$$
By Lemma \ref{unipotent} we may assume that one of $\alpha $ or $\b$ are nonzero. We can write $\eO_2 = \zO_{12} \cup \zO_{13} \cup \zO_{23}$ as a disjoint union, where 
$$\begin{aligned}
\zO_{12} &= \{ v\in \eO_2 : v_1v_2 \ne 0\},\\ \zO_{13} &= \{ v\in \eO_2  : v_2 = 0, v_1v_3 \ne 0\},\\ \zO_{23} &= \{ v\in \eO_2  : v_1 = 0, v_2v_3 \ne 0\}.
\end{aligned}
$$
Note that if $\alpha = 0$, then $\zO_{23} = \emptyset $ and if $\b = 0$ then $\zO_{13} = \emptyset.$  Topological sections in each layer  respectively are 
 $$
 \Sigma_{12} = \{ [\pm 1, \pm 1, v_3 ] : v_3 \in \R\}, \ \Sigma_{13} = \{ [\pm 1, 0 \pm 1]\}, \ \Sigma_{23} = \{ [0,\pm1,\pm1]\}.
 $$
 We insert a lemma for this particular case.

\begin{lemma} Suppose that $\h$ has the form above. Then the following are equivalent.

\begin{itemize}

\item[(i)] $\eO_2/H^T$ is compact.

\item[(ii)] $\eO_2/H^T$ contains a compact open conull subset.

\item[(iii)] At least one of $\alpha$ or $\b$ is positive. 

\end{itemize}

\end{lemma}

\begin{proof} (i) implies (ii) is trivial. To prove (ii) implies (iii), suppose that $\mathcal U$ is an $H^T$-invariant open conull subset of $\eO_2$. Then $\mathcal U \cap \zO_{12} \ne \emptyset$, and it follows that there is an open conull subset $Z$ of $\R$ such that $ \{ \pm 1\} \times \{\pm 1\} \times Z \subset \Sigma_{12} \cap \mathcal U$. Suppose that neither $\alpha$ nor $\b$ are positive, and hence one is negative. Assume that $\b < 0$. Put 
$$
v(k) = [1,1,z(k)]
$$
where $z(k) \in Z$ and $z(k) \rightarrow \infty$. We claim that no subsequence of $(H^Tv(k))$ is convergent. Suppose the claim is false; let $s_j, t_j \in \R$ and $v(k_j) = v(j)$ be a subsequence that converges to a point $v$ in $\eO_2$. Since $z(j) \rightarrow \infty$ and $\Sigma_{12}$ is a topological section for $\zO_{12}$, then $v \notin\zO_{12}$. Hence we may assume that $v \in \Sigma_{13} \cup \Sigma_{23}$. Now we have 
$$
[ e^{s_j}, e^{t_j} , e^{\alpha s_j+\b t_j} z(j)] \rightarrow v.
$$
 with $v \in \Sigma_{13} \cup \Sigma_{23}$. If $v \in \Sigma_{13}$, then $v = [1, 0 ,\pm 1]^t$. Hence  $e^{s_j } \rightarrow 1$ and $e^{t_j } \rightarrow 0$, and so  $e^{\alpha s_j + \b t_j} \rightarrow+\infty $ hence $e^{\alpha s_j + \b t_j} z(j) \rightarrow +\infty$, contradicting the convergence above. If $v \in \Sigma_{23}$, then $v = [0, 1,\pm 1]^t$. Hence  $e^{s_j } \rightarrow 0$ and $e^{t_j } \rightarrow 1$. Now if $\alpha < 0$ then $e^{\alpha s_j + \b t_j}  \rightarrow +\infty $ while if $\alpha = 0$, then $e^{\alpha s_j + \b t_j}  \rightarrow 1$. In either case $e^{\alpha s_j+\b t_j} z(j) \rightarrow +\infty$, contradicting the convergence above. If $\b = 0$ then $\alpha < 0$ and the proof is similar. Thus (ii) implies (iii). 
 
 Now suppose (iii). Let $v(k)$ be any sequence in $\eO_2$. In order to show convergence of a subsequence of orbits, we may first assume that each $v(k)$ belongs to one of the three sections above, and then by passing to a subsequence, that every $v(k)$ belongs to one of these sections. Now if $v(k) \in \Sigma_{13}$ for all $k$, or $v(k) \in \Sigma_{23}$ for all $k$, then since these are finite sets, it is obvious there is a convergent subsequence. Suppose then that $v(k)$ belongs to $\Sigma_{12}$ for all $k$. 
 
 Thus $v(k) = [\pm 1, \pm 1, v_3(k)]$ for each $k$. If the sequence $v_3(k)$ of real numbers is bounded, then again it is clear that we have a convergent subsequence, so, passing to a subsequence, we may assume that $v(k) \rightarrow +\infty$ or  $v(k) \rightarrow -\infty$.  Now suppose that $\b > 0$. Then we can choose $t_k\in \R, k = 1, 2, 3, \dots $ so that 
 $$
 e^{\b t_k} v_3(k) \rightarrow \pm 1.
 $$
 Then $e^{t_k} \rightarrow 0$ and so 
 $$
 e^{ t_k B } v(k) \rightarrow [\pm 1, 0, \pm 1].
 $$
 Similarly, if $\alpha > 0$, then we can choose $s_k \in \R,  k = 1, 2, 3, \dots $ so that 
 $$
e^{ s_k A} v(k) \rightarrow [0,\pm 1, \pm 1].
 $$
 This completes the proof. 
\end{proof}

Having obtained the preceding lemma, the characterization of the cases admitting open compact conull subsets of the dual orbit space is complete. Furthermore, we already obtained a topological section for the case (a), and the finite open orbit cases (c) through (e) are also taken care of. Hence it remains to take a closer look at the existence of quasi-sections for case (b). 

If $\alpha = 0$, one easily verifies that the set $\Sigma = \{ v \in \mathbb{R}^3: v_1 = 1, v_2^2+v_3^2 = 1 \}$ is a topological section; and the case $\beta = 0$ is similar. 

Hence it remains to consider the case $\alpha \beta \not= 0$; w.l.o.g. $\alpha>0$. Our assumptions imply that
\[
 \eO_2 = \{ v \in \mathbb{R}^3~:~ |v_i| \not= 0 \mbox{ for at least two indices } i \}~.
\] Hence, letting for $i =1,2,3$ and $\rho > 1$ 
\[
 C_i(\varrho) = \{ v \in \mathbb{R}^3: |v_i| \le \varrho~, \forall j \not=i: ~\frac{1}{\varrho} \le |v_j| \le \varrho \}~,
\] defines a compact set $C_i(\varrho) \subset \eO_2$. It is therefore sufficient, in view of Lemma \ref{lem:quasi_sect}, to prove that $((C_i(\varrho), C_k(\varrho)))$ is not relatively compact, for some pair $(i,k)$ of indices. 

First assume that $\beta>0$, and consider $((C_1(\varrho), C_2(\varrho)))$. We then need to determine the set of pairs $(s,t) \in \mathbb{R}^2$ such that, with suitable $v \in \mathbb{R}^3$, the following inequalities hold:
\[
\begin{array}{ccc}|v_1| \le \varrho~, & ~\frac{1}{\varrho} \le |v_2| \le \varrho~,~ & ~\frac{1}{\varrho} \le |v_3| \le \varrho \\
\frac{1}{\varrho} \le e^s |v_1| \le \varrho ~, & ~ e^t|v_2| \le \varrho~,~ & ~\frac{1}{\varrho} \le e^{\alpha s + \beta t}|v_3| \le \varrho~.
\end{array}
\]
These inequalities are solvable with suitable $v$ iff 
\[
 -2 \log(\varrho) \le s~,~ t \le 2 \log(\varrho) ~,~ -2 \log(\varrho) \le \alpha s + \beta t \le 2 \log(\varrho)~. 
\]
Given any $s \ge 0$, we may pick $t = - \alpha s/\beta <0 \le 2 \log(\varrho)$ so that $\alpha s + \beta t = 0$. Hence all three inequalities are satisfied, which means that there exists $v \in C_1(\varrho)$ such that $\exp(sA + tB) v \in C_2(\varrho)$, i.e. $\exp(sA + tB) \in ((C_1(\varrho), C2(\varrho)))$. In particular, $((C_1(\varrho), C2(\varrho)))$ is not relatively compact. 

For the case $\beta<0$, we first note that a similar reasoning as above leads to the following necessary and sufficient conditions for $\exp(sA +tB)$ to be in $((C_2(\varrho), C_3(\varrho)))$:
\[
-2 \log(\varrho) \le s \le 2 \log(\varrho)~,~-2 \log(\varrho) \le  t  ~,~ \alpha s + \beta t \le 2 \log(\varrho)~. 
\] This can be fulfilled by picking $s=0$ and $t \ge 0$ arbitrary, showing that $((C_2(\varrho), C_3(\varrho)))$ is not relatively compact. 

Again, Theorem \ref{thm:main} applies to yield non-integrability whenever $\mathcal{O}_{2}/H^T$ contains no nonempty open compact set. 
\end{proof}

\section{Combining diagonal and nilpotent generators}

In this section we consider a particular class of abelian dilation groups. Recall that a one-parameter group infinitesimally generated by a diagonal matrix with strictly positive entries allows integrable admissible vectors. In this section we investigate what happens when we include a second infinitesimal generator, assumed to be nilpotent.  The following result shows that this destroys integrability when
$n>2$.

\begin{prop}\label{not compact} Let $H = \exp \h$, where $\h$ is the real span of two non-zero commuting endomorphisms of $\R^n$, one of which is nilpotent, and the other diagonalizable with real eigenvalues. Suppose that $n \ge 3$, and let $U$ be an $H$-invariant, open, conull subset of $\eO_2$. Then $U / H^T$ is not compact. It follows that there are no integrable admissible vectors for $H$.

\end{prop}

Observe that if $n = 2$ and $H$ is as in the proposition, then $\h $ is the real span of $X$ and $A$ where
$$
X = \left[\begin{matrix} 0&\\1&0\end{matrix}\right], \ \ \ A =  \left[\begin{matrix} 1&\\ &1\end{matrix}\right].
 $$
In this case $\eO_2 = \R^2 \setminus \{ v : v_1 = 0\} $ is the union of two open orbits, and $\eO_2 / H$ is compact.

Recall that $H$ acts on $\widehat{\R^n}$ by $h \cdot \xi = (h^T)^{-1}\xi$, when $\widehat{\R^n}$ is identified with $\R^n$. We make this identification and describe $\h$ in terms of the transposes of its elements. 
Fix $A \in \h$ diagonalizable with real eigenvalues, and write $\R^n = \oplus_{\l \in\mathcal E} W_\l$ where $\mathcal E$  is the set of distinct eigenvalues of $A$, and $W_\l$ is the eigenspace for the eigenvalue $\l$.

Given an eigenspace $W$, put $\tilde W = \oplus \{W_\l : W_\l \ne W\}$, and let $v^W$ be the projection of $v \in \R^n$.
Fix a non-zero nilpotent element $X \in \h^T$. 
Each $W_\l$ is $X$-invariant, and we choose a basis $\{e_1^{(\l)}, \dots , e_{n_\l}^{(\l)}\}$ for $W_\l$ consisting of $\l$-eigenvectors of $A$, and for which the restriction of $X$ to $W_\l$ is of the form
$$
X|_{W_\l} = \left[\begin{matrix} 0&&&\cdots &0 \\
\epsilon_2&0&&\cdots &0 \\
&\epsilon_3&0&\cdots &0 \\
&&\ddots\ddots &&\vdots\\
&&&\epsilon_{n_\l}& 0 \end{matrix}\right]
$$
and $\epsilon_i = 0$ or $1$. For each $\l \in \mathcal E$ and $1 \le i \le n_\l$ let $p^{(\l)}_i : \R^n \rightarrow \R$  denote the coordinate function dual to the basis element $e^{(\l)}_i$. When $W = W_\l$ is fixed we denote $ p^{(\l)}_i$ by $p_i$. 
The following shows that we can assume that all eigenvalues are non-zero.

\begin{lemma} Suppose that $A$ has 0 as an eigenvalue. Then $\R^n / H^T$ contains no nonempty compact open subset. 

\end{lemma}

\begin{proof} Let $T = W_0$ and $T_1 $ the span of $e_2^{(0)}, \dots , e_{n_0}^{(0)}$ in $T$. Let $\pi : \R^n \rightarrow T / T_1$ be the canonical map defined by $\pi(v) = v^{(\l)} + T_1$. Then $\pi$ is continuous and open. Now $T_1$ is $H$-invariant,  and the natural action of $H$ on $T / T_1$ defined by $\exp(sA + tX) \pi(v) = \pi(\exp (sA + tX)v) $ is trivial. Thus $\overline{\pi}(Hv) = H\pi(v) = \pi(v)$ is a well-defined map $\overline{\pi} : \R^n / H \rightarrow T/T_1$ that is continuous and open. Let $K \subset \R^n / H^T$ be compact and open. Then  $\overline{\pi}(K)$ is a compact and open subset of $T/T_1 \simeq \R$ and hence empty. Therefore $K$ is empty. 
\end{proof}

For the remainder of this section we assume that $0 \notin \mathcal E$. For each eigenspace $W$ let $\eO_2(W) = \{ v \in \R^n : \dim H v^W= 2\}.$ We have the following.

\begin{cor}  \label{2d} One has 
$$
 \eO_2 = \cup_{\l \in \mathcal E} \eO_2(W_\l).
 $$

\end{cor}

\begin{proof} Let $v \in \eO_2$. Then necessarily $Xv \ne 0$. Hence for some eigenspace $W = W_\l$, $W \not\subset \ker X$.  Let $b = \min \{ 2 \le i \le n_\l : p_i(Xw)\ne 0\}$. Then 
\begin{equation} \label{action1}
p_{b-1}(e^{s A + tX} v )  = e^{\l s} p_{b-1}(v)
\end{equation}
and 
\begin{equation} \label{action2}
 p_b(e^{s A + tX} v) = e^{\l s} (p_b(v)  +t p_b(Xv).
\end{equation}
Since $\l \ne 0$ then $\dim H w = 2 $ and $v \in \eO_2(W)$.

\end{proof}

Fix an eigenspace $W = W_\l$  for which $W\not\subset \ker X$. 
Recall that we have chosen a basis for $W$ for which $X|_{W}$ has the form
$$
X|_{W} = \left[\begin{matrix} 0&&&\cdots &0 \\
\epsilon_2&0&&\cdots &0 \\
&\epsilon_3&0&\cdots &0 \\
&&\ddots\ddots &&\vdots\\
&&&\epsilon_{n_\l}& 0 \end{matrix}\right]
$$
and not all $\epsilon_i$ are zero. Note that $p_1(Xv) = 0$ and 
\begin{equation}\label{pi}
p_i(Xv) = \epsilon_i p_{i-1}(v), 2 \le i \le n_{\l}.
\end{equation}

A stratification of orbits in $\eO_2(W)$ can be defined as follows. For each $v \in \eO_2(W)$ put
$$
b({v^W}) = \min \{ 2 \le i \le n_\l : p_i(X v) \ne 0\} .
$$
For $2 \le b \le n_\l$ put $\O_b = \{ v \in \R^n : b({v^W})  = b\}$, and let $B = \{ 2 \le i \le n_\l : \epsilon_i \ne 0\}$. Then for each $1 \le b \le n_\l$, the relation (\ref{pi}) shows that $\O_b$ is non-empty if and only if $b \in B$, and the relations (\ref{action1}) and (\ref{action2}) show that $\O_b$ is included in $\eO_2(W)$. Thus we have 
 $$
 \eO_2(W) = \cup_{b \in B} \  \O_b.
 $$
If $v \in \O_b$, then $p_b(X^pv) = $ for all $p \ge 2$, and hence for any $t \in \R$, 
 $$
p_{i}(X e^{ tX} v) = 0 \text{ for all } i < b, \text{ and }  p_b(X e^{ tX} v) = p_b(Xv) \ne 0.
 $$
 Thus $\O_b$ is $e^{ \R X}$-invariant. A section for the $e^{ \R X}$-orbits in $\O_b$ is given as follows.

 \begin{lemma} The set 
$$
\L_b = \O_b \cap \{ v \in \R^n : p_b(v)= 0\}
$$
 is a topological section for the $e^{ \R X}$ orbits in $\O_b$.

 \end{lemma}

 \begin{proof} We have
 $$
p_b(e^{ tX} v) = p_b(v) + t p_b(Xv)= p_b(v) + t p_{b-1}(v)
$$
holds for all $v \in \O_b$. Let $v \in \O_b$, and define
 $$
 t = -\frac{ p_b(v)  }{ p_b(Xv)}. 
 $$
Then $e^{ tX} v \in \L_b$. If $v \in \L_b$ and also $e^{ tX} v \in \L_b$, then 
$$
0 = p_b(e^{ tX} v) = p_b(v)  + t p_b(Xv)= t \ p_b (Xv)
$$
implies $t = 0$ and $e^{ tX} v = v$. 

\end{proof}



Of particular interest is the``minimal layer" $\O_a$, where $a = \min B$. 
Note that $\O_a = \{ v \in \R^n : p_a(Xv) \ne 0\}$
is open and conull in $\R^n$.




\begin{lemma} \label{section} Let $b \in B$. Then $\O_b$ is $H$-invariant, and the set 
$$
\S_b = \{ v \in \L_b : p_b(Xv)  = \pm 1\}
$$
is a topological section for the $H$-orbits in $\O_b$.

\end{lemma}

\begin{proof} Let $v \in \O_b$. For $b' \in B, $ if $b'< b$ then we have $p_{b'}(Xe^{sA}v) = e^{\l s} p_{b'}(Xv) = 0$, while $p_b(Xe^{sA} v) = e^{\l s} p_b(Xv) \ne 0$. Hence $e^{sA} v \in \O_b$ and $\O_b$ is $H$-invariant. To see that $\S_b$ is a section, choose $t \in \R $ as above so that $e^{tX} v \in \L_b$, and put  $s = - \ln |p_b(X e^{tX}  v)|$. Then $e^{sA + tX} v \in \L_b$ and 
$$
|p_b(Xe^{sA + tX}  v |= | e^{\l s} p_b(Xe^{tX} v)| = e^s | p_b(X e^{tX} v| = 1.
$$ 
so $e^{sA + tX} v \in \S_b$ and thus $H\S_b = \O_b$. If $v \in \S_b$ and $e^{sA + tX} v \in \S_b$ also, then the preceding calculations show that $t = s = 0$. 

\end{proof}

Note that by Corollary \ref{2d},  each $v \in \eO_2$ belongs to some layer $\O_b \subset \eO_2(W)$. 
On the other hand it is possible that $\eO_2$ coincides with a minimal layer $\O_a$:  in this case there is only one eigenspace $W$ that is not included in the kernel of $X$, and for the matrix of $X$ in $W$, only one subdiagonal entry $\epsilon_a $ is non-zero.  Applying Lemma \ref{section} we obtain the following.

\begin{cor} \label{t2layer} Suppose that $n \ge 3$ and $\eO_2 = \O_a$. Then $\eO_2/H^T$ does not contain a nonempty compact open subset. 

\end{cor}

\begin{proof} Since $\S_a$ is a topological section for $\eO_2/H^T$, then its components are homeomorphic with the components of $\eO_2 / H$, and any closed and open subset of $\eO_2 / H $ is homeomorphic with a closed and open subset of $\S_a$. Any non-empty closed and open subset of $\S_a$ must contain a component. Now since $\O_a$ is the minimal layer then the condition $p_a(Xv) = \pm 1$ implies that $v \in \O_a$. Thus by Lemma \ref{section} we have
$$
\S_a = \{ v \in V_a : p_a(Xv) = \pm 1\}
$$
and $\S_a$ is the union of two components. Since $n \ge 3$, then each component is an unbounded affine hyperplane in $V_a$. Hence any non-empty closed and open subset of $\eO_2/H^T$ is not compact. But $\eO_2/H^T$ is Hausdorff, so any compact open subset must be a closed and open subset that is bounded, hence must be empty. 

\end{proof}

We are now ready to prove the proposition.

\begin{proof} [Proof of Proposition \ref{not compact}] Let $U$ be an open $H$-invariant subset of $\eO_2$ that is conull for Lebesgue measure. Let $W$ be an eigenspace that is not included in $\ker X$ and $\O = \O_a$ be a minimal layer in $\eO_2(W)$ constructed above. Then $U \cap \O$ is open, conull, and $H$-invariant. If  $\eO_2 = \O$, then by Corollary \ref{t2layer} we are done, so we suppose that $\eO_2$ contains other layers $\O_b$. We scale the generator $A$ so that the eigenvalue for the eigenspace $W$ is 1.

Let $\S = \S_a$ be the topological section for $\O$, so that  $\O = \{ v \in \R^n : p_a(X v) \ne 0\}$ and $\S = \{ v \in V_a : p_a(Xv) = \pm 1, p_a(v) = 0\} $. Put $\S_+ = \{ v \in \S : p_a(Xv) = 1\}$; by the relation (\ref{pi}) we have 
$$
\S_+ = \{ v \in V_a : p_{a-1}(v) = 1, p_a(v) = 0\}.
$$
Put $q = \dim \tilde W$. We identify $\S_+$ with $\R^{a-2} \times \R^{n_1 - a } \times \R \times \R^q$, via the coordinate functions $p_1, \dots p_{a-2}, p_{a+1} \dots , p_{n_1-1}$ and $p_{n_1}$ for $W$, and the coordinate functions for all other $W' \ne W$. Denote the corresponding $n-2$ dimensional Lebesgue measure  $m_{a-2} \times m_{n_1-a-1} \times m_1 \times m_q$ on $\S_+$ by $m $. (If $a = 2$, then $m_{a-2}$ is point mass measure.)


Now $U \cap \O_+$ is conull in $\O_+$, where $\O_+ = H \S_+$, and the measure $m $ is equivalent with the push forward of Lebesgue measure on $\O_+ = \{ v \in \O : p_a(Xv) > 0\}$. Hence  $m( \S_+ \setminus U) = 0$. 
It follows that there is $x \in \R^{a-2}, y \in \R^{n_1-a-1}$, and $w \in \R^q$, such that 
$$
(\S_+ \cap U)^{(x,y, w)} = \{ z \in \R : (x,y,z, w) \in \S_+ \cap U\}
$$
is $m_1$-conull in $\R$. 
Hence there is a sequence $z(j) \in (\S \cap U)^{(x,y,  w)}$ such that $z(j) \rightarrow +\infty$. Define the sequence $v(j)$ in $\S_+$ using the coordinates indicated above by 
$$
v(j) = (x,y,z(j), w), \ j = 1, 2, 3, \dots.
$$
Note that all coordinates of $v(j)$ are fixed except $p_{n_1}(v(j)) = z(j)$, and the elements $x,y$ and $w$ determine all other coordinates of $v(j)$ except $p_{a-1}(v(j)) = 1, p_a(v(j)) = 0$. 
We claim that $Hv(j)$ has no convergent subsequence in $\eO_2 / H$.

Suppose the contrary, that for some $v \in \eO_2$, and for some subsequence $Hv(k)$ of $Hv(j)$,  $H v(k) \rightarrow H v$. Since $z(k) \rightarrow +\infty$ then $v(k)$ is not convergent in $\S$ and $\S$ is a topological section in $\O$, therefore  $Hv \notin \O/H^T$. By definition of the quotient topology on $\eO_2 / H^T$, we have $s(k) , t(k) \in \R, k = 1, 2, 3, \dots $ such that 
$$
v'(k) :=  e^{ s_kA + t_k X }v(k) \rightarrow v.
$$
Now $v \notin \O$ means that $p_a(Xv) = 0$ and so 
\begin{equation} \label{a cond}
p_a(X v'(k)) \rightarrow 0.
\end{equation}
Recall that $p_a(X e^{t_kX} v(k) = p_a(Xv(k)$ and $v(k) \in \S_a$, hence $|p_a(X  e^{t_kX}  v(k))| = 1$. Hence
$$
p_a(X v'(k))  = e^{s_k} \ p_a\bigl( Xe^{t_kX}  v(k)\bigr) = e^{s_k } p_a(Xv(k)) = \pm e^{s_k}
$$
and so (\ref{a cond}) implies
$$
e^{s_k }  \rightarrow 0.
$$
For some eigenspace $W = W_\l$ and layer $\O_b \subset \eO_2(W_\l)$, we have $v \in \O_b$; in fact we may assume that $v \in \S_b$. Hence we have
\begin{equation}\label{vcoord}
p_b(v) = 0, \ \ p_b(Xv) = \pm 1.
\end{equation}
Suppose that $\l< 0$. Then  $e^{\l s_k} \rightarrow \infty$, but 
$$
e^{\l s_k} p_2(Xv(k)) = p_2(Xv'(k))  \rightarrow p_2(Xv). 
$$
Hence $p_2(Xv(k)) = p_2(Xv) = 0$. Now if $b = 2$, then this would already the desired contradiction in light of (\ref{vcoord}). Otherwise, $p_3((X^p v(k)) = 0$ for all $k$ and $p \ge 1$, so now 
$$
e^{\l s_k}p_3(Xv(k)) = p_3(Xv'(k))  \rightarrow p_3(Xv)
$$
and if $b = 3$ we arrive at a contradiction. Continuing in this way we obtain $p_b(Xv) = 0$ while also $p_b(Xv) = \pm 1$, a contradiction.

Now suppose that $\l > 0$, so that $e^{\l s_k} \rightarrow 0$.  Since $p_b(v) = 0$ then
\begin{equation}\label{b0 cond}
p_b(v'(k)) \rightarrow 0
\end{equation}
and since $|p_b(Xv)| =  1$, then 
\begin{equation}\label{b1 cond}
|p_b(X v'(k))| \rightarrow 1.
\end{equation}
We show that the preceding relations are not compatible. First, rewrite  (\ref{b0 cond}),  (\ref{b1 cond})  as
\begin{equation} \label{b2 cond}
e^{\l s_k } \ p_b(e^{t_kX} \ v(k)) \rightarrow 0
\end{equation}
and 
\begin{equation}\label{b3 cond}
e^{\l s_k }  |p_b\bigl( Xe^{t_kX}  v(k)\bigr)| \rightarrow 1.
\end{equation}
To compute $p_b(e^{t_kX}  v(k))$ and $p_b\bigl(X  e^{t_kX}  v(k)\bigr)$, note that
for each $p\ge 1$,   $p_b(X^pv(k)) $ is independent of $k$, and put $c_p=p_b(X^pv(k))$. 
Then
$$
\begin{aligned}
p_b\bigl(X e^{t_kX}  v(k)\bigr) &= p_b \bigl( Xv(k)  +  t_k X^2 v(k) + \frac{1}{2!} t_k^2X^3v(k) + \cdots  \bigr)\\
& = c_1 +c_2 t_k+ \frac{c_3}{2! } t_k^2 + \cdots  
\end{aligned}
$$
and since $e^{s_k } \rightarrow 0$, then the relation (\ref{b3 cond}) shows that 
$$
\left|c_1 +c_2 t_k + \frac{c_3}{2! } t_k^2 + \cdots \right| 
 \longrightarrow \infty.
$$
This shows that not all $c_p$ are zero and $|t_k| \rightarrow +\infty$. 
Let $N = \max \{ p  : c_p \ne 0\}$; then $N\ge 1$ and we have
$$
e^{s_k} \left| c_1+ c_2 t_k + \cdots  +  \dfrac{c_{N}}{(N-1)!}  t_k^{N-1}\right| \longrightarrow 1
$$
so
$$
\begin{aligned}
e^{s_k }|t_k|^N &\left| \dfrac{c_1}{t_k^{N-1}}+ \dfrac{c_2}{ t_k^{N-2}} + \cdots  +  \dfrac{c_{N}}{(N-1)!} \right| \\
&= |t_k| \ e^{s_k } \left| c_1+ c_2 t_k + \cdots  +  \dfrac{c_{N}}{(N-1)!}  t_k^{N-1}\right| \longrightarrow +\infty
\end{aligned}
$$
hence
$$
e^{s_k }|t_k|^N \longrightarrow +\infty.
$$
Similarly we compute that
$$
p_b\bigl(e^{t_kX}  v(k)\bigr) = 
c_0 + c_1 t_k+ \frac{c_2}{2! } t_k^2 + \cdots +  \dfrac{c_N}{N!}  t_k^{N}.
$$
But now the relation (\ref{b2 cond}) gives
$$
\begin{aligned}
e^{s_k}|t_k|^N \left| \dfrac{c_0}{t_k^N}+ \dfrac{c_1}{ t_k^{N-1}} + \cdots  +  \dfrac{c_{N}}{N!}\right| 
&= e^{s_k } \left| c_0 + c_1 t_k + \cdots  +  \dfrac{c_N}{N!}  t_k^{N}\right|  \\
&= e^{s_k} |p_b\bigl(e^{t_kX}  v(k)\bigr) |\longrightarrow 0
\end{aligned}
$$
and hence 
$$
e^{s_k }|t_k|^N \longrightarrow 0
$$
a contradiction.

\end{proof}

\section*{Acknowledgements}
Brad Currey would like to thank RWTH Aachen University for their hospitality during his visit. Hartmut F\"uhr would like to thank Dalhousie University for their hospitality, and Karlheinz Gr\"ochenig for interesting discussions
on the topic.

\end{document}